\documentclass[12pt]{amsart}
\usepackage{latexsym,fancyhdr,amssymb,color,amsmath,amsthm,graphicx,listings,comment}
 
\newtheorem{thm}{Theorem}       
       \newtheorem{coro}{Corollary}
\usepackage[section]{placeins}      \let\paragraph\subsection
\newcommand\scalemath[2]{\scalebox{#1}{\mbox{\ensuremath{\displaystyle #2}}}}

\title{The counting matrix of a simplicial complex}
\author{Oliver Knill} \date{7/21/2019}
\address{Department of Mathematics \\ Harvard University \\ Cambridge, MA, 02138 }

\begin{document}

\begin{abstract}
For a finite abstract simplicial complex $G$ with $n$ sets, define
the $n \times n$ matrix $K(x,y)=|W^-(x) \cap W^-(y)|$ which is the number
of subsimplices in $x \cap y$. We call it the counting matrix of $G$. 
Similarly as the connection matrix
$L$ which is $L(x,y)=1$ if $x$ and $y$ intersect and $0$ else, the counting matrix
$K$ is unimodular. Actually, $K$ is always in $SL(n,Z)$. The inverse of $K$ has the Green function
entries $K^{-1}(x,y)=\omega(x) \omega(y) |W^+(x) \cap W^+y|$, where $W^+(x)$ is
the star of $x$, the sets in $G$ which contain $x$. The matrix $K$ is always positive
definite. The spectra of $K$ and $K^{-1}$ always agree so that the matrix 
$Q=K-K^{-1}$ has the spectral symmetry $\sigma(Q)=-\sigma(Q)$ and 
the zeta function 
$\zeta(s) = \sum_{k=1}^n \lambda_k^{-s}$ defined by the eigenvalues $\lambda_k$ of $K$ 
satisfies the functional equation $\zeta(a+ib)=\zeta(-a+ib)$. 
The energy theorem in this case tells that
the total potential energy is $\sum_{x,y} K^{-1}(x,y)=|G|=\sum_x 1$ is
the number sets in $G$. In comparison, we had in the connection matrix case
the identity $\sum_{x,y} L^{-1}(x,y)=\chi(G)=\sum_x \omega(x)$. 
\end{abstract} 

\maketitle

\section{The results}

\paragraph{}
The category of {\bf finite abstract simplicial complexes} introduced by Dehn and Heegaard
\cite{DehnHeegaard,BurdeZieschang} 
is astonishingly rich \cite{AmazingWorld}, despite of the minimal axiomatic set-up:
$G$ is a finite set of sets closed under the operation of taking non-empty sub-sets. 
The structure is abundant as it appears for clique or independence complexes of graphs, 
order complexes of finite posets or  matroids without the empty set. Some authors like
\cite{JonssonSimplicial} include the empty set.
We use the more topological framework without empty set as used for example in 
\cite{FerrarioPiccinini,MunkresAlgebraicTopology}. 

\paragraph{}
The energy theorem for finite abstract simplicial complexes \cite{EnergyComplex}
states that the sum of the Green function entries of the matrix $L^{-1}(x,y)$ is the
{\bf Euler characteristic} $\chi(G) = \sum_x \omega(x)$, where 
$\omega(x) = (-1)^{|x|}$. The matrix entries of $L$ were defined as 
$L(x,y) = \chi( W^-(x) \cap W^-(y)))$, where $W^-(x)$ is the {\bf core} of $x$, 
the simplicial complex obtained as the set of sets contained 
in $x$. The inverse of $L$ has the matrix entries 
$\omega(x) \omega(y) \chi(W^+(x) \cap W^+(y))$, where $W^+(x)$ is the {\bf star}, the 
set of sets in $G$ containing $x$. These Green function entries can be seen as
a potential between $x$ and $y$. 

\paragraph{}
Similarly as when replacing the ``super counting" the permutations in Leibniz' definition of
the determinant to ``counting", leading to permanents, one can look what happens if one replaces 
the super-counting in the connection case with counting. Define the {\bf counting matrix}
$$  K(x,y) = |W^-(x) \cap W^-(y)| = 2^{|x \cap y|}-1   \; . $$  
Like the {\bf connection matrix} $L(x,y)$ of $G$, it is non-negative matrix defined by $G$. 
But it is no more a 0-1 matrix. While in the connection matrix case, 
${\rm det}(L)=\prod_{x \in G} \omega(x)$, we have now ${\rm det}(K) = \prod_{x \in G} 1 =1$
which is expressed in: 

\begin{thm}[Unimodularity]
$K \in SL(n,Z)$. 
\end{thm}
\begin{proof}
(Sketch) Similarly as in the case of the matrix $L$, it is convenient to prove the result in 
a more general class of discrete CW complexes, where adding a new cell adds a single new 
row and column to the matrix $K$. This allows induction. If there $n$ cells in $G$ and the
new CW complex is $G+_A x$, then look at the matrix
$$ K(t) = \left[ \begin{array}{ccccccc}
    K_{11} & K_{12} &     . & . & . & K_{1n} & t K_{1,x} \\
    K_{21} & K_{22} &     . & . & . & .      & t K_{2,x} \\
      .    & .      &     . & . & . & .      & t K_{3,x} \\
      .    & .      &     . & . & . & .      & .         \\
      .    & .      &     . & . & . & .      & .         \\
      .    & .      &     . & . & . & .      & .         \\
    K_{n1} & .      &     . & . & . & K_{nn} & t K_{n,x} \\
    K_{x1} & K_{x2} &\dots  & \dots & \dots  & K_{xn} & |W^-(x)|  
            \end{array} \right]  \; .   $$
For $t=0$ this has by induction the determinant $|W^-(x)|$. 
For $t=1$, we have the counting matrix of $G+_A x$. Computing the 
determinant with a Laplace expansion with respect to the last column
gives a linear expression $at+b$ in $t$, where $b=|W^-(x)|$. 
The slope is $|W^-(x)|-(|W^-(x)|-1)$ because $A$ has $|W^-(x)|-1$ elements
which each contribute $1$ as one can see when taking the derivative with respect
to $t$ at $t=0$, where it is a determinant of a smaller counting matrix.
Alternatively, an adaptation of the proof given in \cite{MukherjeeBera2018} 
should go through. 
\end{proof} 

\paragraph{}
For a set $x \in G$, the {\bf star} $W^+(x)$ of $x$ is the set of simplices which contain $x$.
It is somehow dual to the {\bf core} $W^-(x)$ which is the set of simplices which are contained in $x$. 
But unlike the core, the star is in general not a simplicial complex. Still, one can count or 
super count elements in a set of sets. 
We have seen that $L^{-1}(x,y) = \omega(x) \omega(y) \chi(W^+(x) \cap W^+(y))$
relates the Green function entries with stars in $G$. 
There is an analogue formula for $K$:

\begin{thm}[Green-Star]
$K^{-1}(x,y) = \omega(x) \omega(y) |W^+(x) \cap W^+(y)|$. 
\end{thm}
\begin{proof}
(Sketch) 
To prove the identity $K^{-1} K = 1$, one has to check two cases: \\
a) given $x \in G$, the sum
$$ \sum_{y \in G} \omega(y) |W^+(x) \cap W^+(y)| |W^-(y) \cap W^-(x)| = \omega(x) \; . $$
This means that
$$ \sum_{y \in G} \sum_{u, x \cup y \subset u} \omega(y) (2^{|x \cap y|}-1) = \omega(x) \; . $$
b) given two different sets $x,z$ in $G$, then
$$ \sum_{y \in G} \omega(y) |W^+(x) \cap W^+(y)| |W^-(y) \cap W^-(z)| = 0 \;  $$
which means 
$$ \sum_{y \in G} \sum_{u, x \cup y \subset u} \omega(y) (2^{|x \cap z|}-1) = 0 \; . $$
\end{proof}

\paragraph{}
Especially, the self interaction energy of a simplex $x$ is $K^{-1}(x,x)=|W^+(x)|$ which 
is the cardinality of the star of $x$. In the connection case, we had 
$L^{-1}(x,x) = \chi(W^+(x))$, the Euler characteristic of the star of $x$. 

\paragraph{}
In the connection case $L$, the eigenvalues had both positive and negative parts and
$\chi(G)$ was the number of positive eigenvalues minus the number of negative eigenvalues. 
This still holds for $K$, but all eigenvalues of $K$ are now positive:

\begin{thm}[Positive definite]
The matrix $K$ is positive definite.
\end{thm}
\begin{proof}
(Sketch)
We use the same general CW setup. We saw that
when adding a new cell $x$, the determinant of $K(t)$ can not change sign as
${\rm det}(K(t)) = |W^-(x)|  - (|W^-(x)|-1) t= t-(1-t) |W^-(x)|$. 
While during the deformation, the previous eigenvalues change, none of them
can cross $0$ and become negative because this would lead to a zero determinant.
\end{proof}

\paragraph{}
We can think of $K$ therefore as a Laplacian. It is a bit special as it has no kernel. 
Whenever one has a Laplacian $K$ 
on a geometry then the Green function entries $K^{-1}(x,y)$ play an important role. 
In the Euclidean space it leads to the Newton potential of gravity of electromagnetism.
Here in the discrete, where singularities are absent, the potential energy is ``quantized", 
and the total energy of a constant measure is computable.
The analog of the energy theorem $\sum_{x,y} L^{-1}(x,y) = \chi(G)$  is now:

\begin{thm}[Energy theorem]
$\sum_{x,y} K^{-1}(x,y) = |G|$. 
\end{thm}
\begin{proof}
(Sketch)
For every $x$, we get a potential by adding up all potential energy 
contributions of other sets 
$$ V(x) = \sum_y K^{-1}(x,y) \;  $$ 
This can now be interpreted as an {\bf index} for the dimension
functional $-{\rm dim}$ on $G$, which is locally injective on the 
graph defined by $G$ in which two sets are connected if one is
contained in the other. Poincar\'e Hopf theorem assures then that
$\sum_x V(x)$ is the valuation under consideration, here $X(G)=|G|$. 
\end{proof} 

\paragraph{}
In the connection matrix case, we had a spectral symmetry $\sigma(L^2) = \sigma(L^{-2})$ but 
only if the complex was one-dimensional, meaning that $G$ does not contain sets of size $3$ or
higher. This led to a functional equation for the zeta function in that case. Now, with having $K$ 
positive definite, we do not need to square the matrix.
In the counting case, things are true for all simplicial complexes:

\begin{thm}[Spectral symmetry]
$K$ and $K^{-1}$ have the same spectrum. 
\end{thm}
\begin{proof}
(Sketch) The spectral symmetry is equivalent to 
the statement that the coefficients of the 
characteristic polynomial of $K$ form a {\bf palindromic}
or {\bf anti-palindromic} sequence. 
The coefficients of the characteristic polynomial are
given in terms of minors. As in the case $L=1$, we
can deform the matrix $K$ with a parameter $s$ so that
for $s=0$, we have already established matrix and for $s=1$
we have the case where a new cell has been attached.
A detailed proof needs the analog of the {\bf Artillery proposition}
in \cite{DyadicRiemann}. 
\end{proof}

\paragraph{}
The above spectral symmetry holds also for {\bf symplectic matrices}. Indeed, a theorem of Kirby assures
that if the spectral symmetry is satisfied and $n$ is even, 
then $K$ is similar to a symplectic matrix. In any case, 
the spectral property is a notion of {\bf reversibility} for the random 
walk defined by $K$. The symmetry is not complete as only $K$ is non-negative,
and its inverse is not. 
Still, there is no symmetry as the core $W^-(x)$ and the star $W^+(x)$ 
are different objects. The former is always a simplicial complex 
generated by one set, while the later can be complicated. 
It can have quite arbitrary Euler characteristic.

\paragraph{}
If $\lambda_k$ are the eigenvalues of $K$, define the {\bf counting zeta function}
$$ \zeta(s) = \sum_{k=1}^n \lambda_k^{-s}  \; . $$
It is an entire function from $\mathbb{C} \to \mathbb{C}$ and unambiguously defined as
$\lambda_k^{-s}= e^{-\log(\lambda_k) s}$ with $\lambda_k>0$ if we take naturally 
the real branch of the logarithm. It immediately follows from the spectral symmetry that
the Zeta function enjoys a functional equation: 

\begin{coro}[Functional equation] 
$\zeta(a+ib) = \zeta(-a+ib)$.
\end{coro}

\paragraph{}
Finally, we can look at the ring $\mathcal{G}$ generated by simplicial complexes. 
This ring has now an other representation in a tensor ring of all matrices 
$SL(Z)=\bigcup_k SL(n,Z)$. 
The disjoint union of two complexes produces the direct sum of matrices
$K(G + H) = K(G)  \oplus K(H)$ and the Cartesian product of two complexes (which is
not a simplicial complex but an element in the ring generated by complexes)
has $K(G \times H ) = K(G) \otimes K(H)$.

\paragraph{}
The empty complex $G=\{ \}$ has the empty counting matrix $K$. It is custom to assign to the 
empty matrix the determinant $1 = 0!$ as it is custom in matrix analysis. Here we have
to address it as the empty complex $0=\{\}$ is a simplicial complex which is the 
zero element in the ring $\mathcal{G}$. The one point complex $\{ \{1\} \}$ with 
$K(1) = 1 \in SL(1,\mathcal{Z})$ is the {\bf one-element} in $\mathcal{G}$. 

\begin{thm}[Representation]
The ring $\mathcal{G}$ has a representation in the tensor ring of all finite
unimodular matrices.
\end{thm}
\begin{proof}
On the matrix level, the matrix $K(G * H)$ is the tensor product
of $K(G)$ with $K(H)$. The counting matrix $L(G \times H)( (a,b), (c,d) )$
is $2^{|a \cap c| + |b \cap d|}  -1)$. 
\end{proof} 

\paragraph{}
While $\mathcal{G}$ is Abelian, the tensor ring is not, as $L(G) \otimes L(H)$
is different from $L(H) \otimes L(G)$. Even the direct sum addition is not commutative,
as $L(G) \oplus L(H)$ places the matrix $L(G)$ first and then $L(H)$. 
However, it is custom to identify similarity classes in the tensor ring. As both
direct sum and tensor product honor the similarity classes, a quotient of the 
tensor ring is a commutative ring with $1$-element. 

\paragraph{} Even the definition of the counting matrix $K(G)$
depends on an {\bf ordering} of the sets in $G$. Obviously, isomorphic complexes
are conjugated by permutation matrices. 
The representation $G \to K(G)$ is injective and if $G \sim H$, 
then $K(G) \sim K(H)$ and if $K(G) \sim K(H)$, then $G \sim H$ because
a matrix $K(G)$ allows a reconstruction of the complex $G$: the diagonal 
entries already determine the dimension, the off diagonal entries indicate then
how big the intersection between two simplices is. 

\paragraph{}
As $K(-G) = -K(G)$ so that only for complexes with an 
even number of sets, $K(-G)$ is also in $SL(n,Z)$. We can restrict
to the smaller ring containing only complexes $G$ with an even number of sets and 
then have a representation in the space of $SL(Z)$ for which all matrices are 
similar to symplectic matrices by Kirby's theorem. 

\begin{coro}[Symplectic representation]
The subring of $\mathcal{G}$ with an even number of sets can be represented in 
in a tensor ring of symplectic matrices.
\end{coro}

\section{Examples}

\paragraph{}
If $G$ is the Whitney complex of the star graph with 3 spikes. It is
$G=\{\{1\},\{2\},\{3\},\{4\},\{1,2\},\{1,3\},\{1,4\}\}$. The counting matrix is 
$$ K = \left[
                  \begin{array}{ccccccc}
                   1 & 0 & 0 & 0 & 1 & 1 & 1 \\
                   0 & 1 & 0 & 0 & 1 & 0 & 0 \\
                   0 & 0 & 1 & 0 & 0 & 1 & 0 \\
                   0 & 0 & 0 & 1 & 0 & 0 & 1 \\
                   1 & 1 & 0 & 0 & 3 & 1 & 1 \\
                   1 & 0 & 1 & 0 & 1 & 3 & 1 \\
                   1 & 0 & 0 & 1 & 1 & 1 & 3 \\
                  \end{array}
                  \right] $$
and its inverse
$$ K^{-1}= \left[
                  \begin{array}{ccccccc}
                   4 & 1 & 1 & 1 & -1 & -1 & -1 \\
                   1 & 2 & 0 & 0 & -1 & 0 & 0 \\
                   1 & 0 & 2 & 0 & 0 & -1 & 0 \\
                   1 & 0 & 0 & 2 & 0 & 0 & -1 \\
                   -1 & -1 & 0 & 0 & 1 & 0 & 0 \\
                   -1 & 0 & -1 & 0 & 0 & 1 & 0 \\
                   -1 & 0 & 0 & -1 & 0 & 0 & 1 \\
                  \end{array}
                  \right] $$
The eigenvalues of $K$ and $K^{-1}$ are 
$$ \left\{3+2 \sqrt{2},\frac{1}{2} \left(3+\sqrt{5}\right),\frac{1}{2}
    \left(3+\sqrt{5}\right),1,\frac{1}{2} \left(3-\sqrt{5}\right),\frac{1}{2}
    \left(3-\sqrt{5}\right),3-2 \sqrt{2}\right\}  \; .  $$

\paragraph{}
Let $G$ be the triangle complex $G=\{ \{1\},\{2\},\{3\},  \{1,2,3\},\{1,2\},\{1,3\},\{2,3\} \}$.
Then 
$$ K = \left[     \begin{array}{ccccccc}
                   1 & 0 & 0 & 1 & 1 & 0 & 1 \\
                   0 & 1 & 0 & 1 & 0 & 1 & 1 \\
                   0 & 0 & 1 & 0 & 1 & 1 & 1 \\
                   1 & 1 & 0 & 3 & 1 & 1 & 3 \\
                   1 & 0 & 1 & 1 & 3 & 1 & 3 \\
                   0 & 1 & 1 & 1 & 1 & 3 & 3 \\
                   1 & 1 & 1 & 3 & 3 & 3 & 7 \\
                  \end{array} \right]  \; . $$
Its eigenvalues are 
$$ \left\{6+\sqrt{35},\frac{1}{2} \left(3+\sqrt{5}\right),\frac{1}{2}
    \left(3+\sqrt{5}\right),1,\frac{1}{2} \left(3-\sqrt{5}\right),\frac{1}{2}
    \left(3-\sqrt{5}\right),6-\sqrt{35}\right\} \; . $$
The inverse of $K$ is 
$$ K^{-1} = \left[
                  \begin{array}{ccccccc}
                   4 & 2 & 2 & -2 & -2 & -1 & 1 \\
                   2 & 4 & 2 & -2 & -1 & -2 & 1 \\
                   2 & 2 & 4 & -1 & -2 & -2 & 1 \\
                   -2 & -2 & -1 & 2 & 1 & 1 & -1 \\
                   -2 & -1 & -2 & 1 & 2 & 1 & -1 \\
                   -1 & -2 & -2 & 1 & 1 & 2 & -1 \\
                   1 & 1 & 1 & -1 & -1 & -1 & 1 \\
                  \end{array}
                  \right] \; . $$
It has the same eigenvalues. 

\paragraph{}
If $G= \{\{1\},\{2\},\{3\},\{4\},\{1,2\},\{1,4\},\{2,3\},\{3,4\}\}$ is the Whitney complex of the 
cycle graph $C_4$, then 
$$ K = \left[
                 \begin{array}{cccccccc}
                  1 & 0 & 0 & 0 & 1 & 1 & 0 & 0 \\
                  0 & 1 & 0 & 0 & 1 & 0 & 1 & 0 \\
                  0 & 0 & 1 & 0 & 0 & 0 & 1 & 1 \\
                  0 & 0 & 0 & 1 & 0 & 1 & 0 & 1 \\
                  1 & 1 & 0 & 0 & 3 & 1 & 1 & 0 \\
                  1 & 0 & 0 & 1 & 1 & 3 & 0 & 1 \\
                  0 & 1 & 1 & 0 & 1 & 0 & 3 & 1 \\
                  0 & 0 & 1 & 1 & 0 & 1 & 1 & 3 \\
                 \end{array}
                 \right] $$
has eigenvalues $\left\{3+2 \sqrt{2},2+\sqrt{3},2+\sqrt{3},1,1,2-\sqrt{3},2-\sqrt{3},3-2 \sqrt{2}\right\}$. 
One might have the impression from those examples that the eigenvalues are solvable expressions, while
algebraic of course they are roots of characteristic polynomial equations which in general can not be 
solved by radicals.

\paragraph{}
Attaching a cell to $C_4$ produces a CW complex $G$ which is topologically a disc. Its $f$-vector
is $(4,4,1)$. It is not a simplicial complex as the last cell does not attach to the boundary 
of a simplex but to a circular graph $C_4$. The new two dimensional cell now contains $9$ subcells.
This is different from a number like $2^{k}-1$ which happens in the case of simplicial complexes. 
We have 
$$ K = \left[
                 \begin{array}{ccccccccc}
                  1 & 0 & 0 & 0 & 1 & 1 & 0 & 0 & 1 \\
                  0 & 1 & 0 & 0 & 1 & 0 & 1 & 0 & 1 \\
                  0 & 0 & 1 & 0 & 0 & 0 & 1 & 1 & 1 \\
                  0 & 0 & 0 & 1 & 0 & 1 & 0 & 1 & 1 \\
                  1 & 1 & 0 & 0 & 3 & 1 & 1 & 0 & 3 \\
                  1 & 0 & 0 & 1 & 1 & 3 & 0 & 1 & 3 \\
                  0 & 1 & 1 & 0 & 1 & 0 & 3 & 1 & 3 \\
                  0 & 0 & 1 & 1 & 0 & 1 & 1 & 3 & 3 \\
                  1 & 1 & 1 & 1 & 3 & 3 & 3 & 3 & 9 \\
                 \end{array} \right]  $$
and 
$$ K^{-1} =     \left[ \begin{array}{ccccccccc}
                   4 & 2 & 1 & 2 & -2 & -2 & -1 & -1 & 1 \\
                   2 & 4 & 2 & 1 & -2 & -1 & -2 & -1 & 1 \\
                   1 & 2 & 4 & 2 & -1 & -1 & -2 & -2 & 1 \\
                   2 & 1 & 2 & 4 & -1 & -2 & -1 & -2 & 1 \\
                   -2 & -2 & -1 & -1 & 2 & 1 & 1 & 1 & -1 \\
                   -2 & -1 & -1 & -2 & 1 & 2 & 1 & 1 & -1 \\
                   -1 & -2 & -2 & -1 & 1 & 1 & 2 & 1 & -1 \\
                   -1 & -1 & -2 & -2 & 1 & 1 & 1 & 2 & -1 \\
                   1 & 1 & 1 & 1 & -1 & -1 & -1 & -1 & 1 \\
                  \end{array} \right]  \; . $$

\paragraph{} Let $G$ be the CW complex, where a new cell is 
attached to an octahedron complex. The counting matrix is
$$ K=\left[
\scalemath{0.8}{
\begin{array}{ccccccccccccccccccccccccccc}
1&0&0&0&0&0&1&1&1&1&0&0&0&0&0&0&0&0&1&1&1&1&0&0&0&0&1\\
0&1&0&0&0&0&1&0&0&0&1&1&1&0&0&0&0&0&1&1&0&0&1&1&0&0&1\\
0&0&1&0&0&0&0&1&0&0&1&0&0&1&1&0&0&0&1&0&1&0&1&0&1&0&1\\
0&0&0&1&0&0&0&0&1&0&0&1&0&0&0&1&1&0&0&1&0&1&0&1&0&1&1\\
0&0&0&0&1&0&0&0&0&1&0&0&0&1&0&1&0&1&0&0&1&1&0&0&1&1&1\\
0&0&0&0&0&1&0&0&0&0&0&0&1&0&1&0&1&1&0&0&0&0&1&1&1&1&1\\
1&1&0&0&0&0&3&1&1&1&1&1&1&0&0&0&0&0&3&3&1&1&1&1&0&0&3\\
1&0&1&0&0&0&1&3&1&1&1&0&0&1&1&0&0&0&3&1&3&1&1&0&1&0&3\\
1&0&0&1&0&0&1&1&3&1&0&1&0&0&0&1&1&0&1&3&1&3&0&1&0&1&3\\
1&0&0&0&1&0&1&1&1&3&0&0&0&1&0&1&0&1&1&1&3&3&0&0&1&1&3\\
0&1&1&0&0&0&1&1&0&0&3&1&1&1&1&0&0&0&3&1&1&0&3&1&1&0&3\\
0&1&0&1&0&0&1&0&1&0&1&3&1&0&0&1&1&0&1&3&0&1&1&3&0&1&3\\
0&1&0&0&0&1&1&0&0&0&1&1&3&0&1&0&1&1&1&1&0&0&3&3&1&1&3\\
0&0&1&0&1&0&0&1&0&1&1&0&0&3&1&1&0&1&1&0&3&1&1&0&3&1&3\\
0&0&1&0&0&1&0&1&0&0&1&0&1&1&3&0&1&1&1&0&1&0&3&1&3&1&3\\
0&0&0&1&1&0&0&0&1&1&0&1&0&1&0&3&1&1&0&1&1&3&0&1&1&3&3\\
0&0&0&1&0&1&0&0&1&0&0&1&1&0&1&1&3&1&0&1&0&1&1&3&1&3&3\\
0&0&0&0&1&1&0&0&0&1&0&0&1&1&1&1&1&3&0&0&1&1&1&1&3&3&3\\
1&1&1&0&0&0&3&3&1&1&3&1&1&1&1&0&0&0&7&3&3&1&3&1&1&0&7\\
1&1&0&1&0&0&3&1&3&1&1&3&1&0&0&1&1&0&3&7&1&3&1&3&0&1&7\\
1&0&1&0&1&0&1&3&1&3&1&0&0&3&1&1&0&1&3&1&7&3&1&0&3&1&7\\
1&0&0&1&1&0&1&1&3&3&0&1&0&1&0&3&1&1&1&3&3&7&0&1&1&3&7\\
0&1&1&0&0&1&1&1&0&0&3&1&3&1&3&0&1&1&3&1&1&0&7&3&3&1&7\\
0&1&0&1&0&1&1&0&1&0&1&3&3&0&1&1&3&1&1&3&0&1&3&7&1&3&7\\
0&0&1&0&1&1&0&1&0&1&1&0&1&3&3&1&1&3&1&0&3&1&3&1&7&3&7\\
0&0&0&1&1&1&0&0&1&1&0&1&1&1&1&3&3&3&0&1&1&3&1&3&3&7&7\\
1&1&1&1&1&1&3&3&3&3&3&3&3&3&3&3&3&3&7&7&7&7&7&7&7&7&27\\ 
\end{array}
} \right] \; . $$
The inverse is 
$$ K^{-1} =  \left[
\scalemath{0.5}{
\begin{array}{ccccccccccccccccccccccccccc}
10&4&4&4&4&1&-4&-4&-4&-4&-2&-2&-1&-2&-1&-2&-1&-1&2&2&2&2&1&1&1&1&-1\\
4&10&4&4&1&4&-4&-2&-2&-1&-4&-4&-4&-1&-2&-1&-2&-1&2&2&1&1&2&2&1&1&-1\\
4&4&10&1&4&4&-2&-4&-1&-2&-4&-1&-2&-4&-4&-1&-1&-2&2&1&2&1&2&1&2&1&-1\\
4&4&1&10&4&4&-2&-1&-4&-2&-1&-4&-2&-1&-1&-4&-4&-2&1&2&1&2&1&2&1&2&-1\\
4&1&4&4&10&4&-1&-2&-2&-4&-1&-1&-1&-4&-2&-4&-2&-4&1&1&2&2&1&1&2&2&-1\\
1&4&4&4&4&10&-1&-1&-1&-1&-2&-2&-4&-2&-4&-2&-4&-4&1&1&1&1&2&2&2&2&-1\\
-4&-4&-2&-2&-1&-1&4&2&2&1&2&2&1&1&1&1&1&1&-2&-2&-1&-1&-1&-1&-1&-1&1\\
-4&-2&-4&-1&-2&-1&2&4&1&2&2&1&1&2&1&1&1&1&-2&-1&-2&-1&-1&-1&-1&-1&1\\
-4&-2&-1&-4&-2&-1&2&1&4&2&1&2&1&1&1&2&1&1&-1&-2&-1&-2&-1&-1&-1&-1&1\\
-4&-1&-2&-2&-4&-1&1&2&2&4&1&1&1&2&1&2&1&1&-1&-1&-2&-2&-1&-1&-1&-1&1\\
-2&-4&-4&-1&-1&-2&2&2&1&1&4&1&2&1&2&1&1&1&-2&-1&-1&-1&-2&-1&-1&-1&1\\
-2&-4&-1&-4&-1&-2&2&1&2&1&1&4&2&1&1&1&2&1&-1&-2&-1&-1&-1&-2&-1&-1&1\\
-1&-4&-2&-2&-1&-4&1&1&1&1&2&2&4&1&2&1&2&1&-1&-1&-1&-1&-2&-2&-1&-1&1\\
-2&-1&-4&-1&-4&-2&1&2&1&2&1&1&1&4&2&1&1&2&-1&-1&-2&-1&-1&-1&-2&-1&1\\
-1&-2&-4&-1&-2&-4&1&1&1&1&2&1&2&2&4&1&1&2&-1&-1&-1&-1&-2&-1&-2&-1&1\\
-2&-1&-1&-4&-4&-2&1&1&2&2&1&1&1&1&1&4&2&2&-1&-1&-1&-2&-1&-1&-1&-2&1\\
-1&-2&-1&-4&-2&-4&1&1&1&1&1&2&2&1&1&2&4&2&-1&-1&-1&-1&-1&-2&-1&-2&1\\
-1&-1&-2&-2&-4&-4&1&1&1&1&1&1&1&2&2&2&2&4&-1&-1&-1&-1&-1&-1&-2&-2&1\\
2&2&2&1&1&1&-2&-2&-1&-1&-2&-1&-1&-1&-1&-1&-1&-1&2&1&1&1&1&1&1&1&-1\\
2&2&1&2&1&1&-2&-1&-2&-1&-1&-2&-1&-1&-1&-1&-1&-1&1&2&1&1&1&1&1&1&-1\\
2&1&2&1&2&1&-1&-2&-1&-2&-1&-1&-1&-2&-1&-1&-1&-1&1&1&2&1&1&1&1&1&-1\\
2&1&1&2&2&1&-1&-1&-2&-2&-1&-1&-1&-1&-1&-2&-1&-1&1&1&1&2&1&1&1&1&-1\\
1&2&2&1&1&2&-1&-1&-1&-1&-2&-1&-2&-1&-2&-1&-1&-1&1&1&1&1&2&1&1&1&-1\\
1&2&1&2&1&2&-1&-1&-1&-1&-1&-2&-2&-1&-1&-1&-2&-1&1&1&1&1&1&2&1&1&-1\\
1&1&2&1&2&2&-1&-1&-1&-1&-1&-1&-1&-2&-2&-1&-1&-2&1&1&1&1&1&1&2&1&-1\\
1&1&1&2&2&2&-1&-1&-1&-1&-1&-1&-1&-1&-1&-2&-2&-2&1&1&1&1&1&1&1&2&-1\\
-1&-1&-1&-1&-1&-1&1&1&1&1&1&1&1&1&1&1&1&1&-1&-1&-1&-1&-1&-1&-1&-1&1\\
\end{array} 
} \right] \; . $$
The sum $\sum_{x,y} K^{-1}(x,y) = 27$. The eigenvalues of $K$
are in the interval $[0.0200446,49.8889]$. There are $7$ eigenvalues $1$ which 
corresponds to roots of the ``super charge" $Q=K-K^{-1}$. The eigenvalues
of 
$\sigma(Q) = \{-2\sqrt{7(45+8\sqrt{30})}$, $2\sqrt{7(45+8\sqrt{30})}$, 
$-8\sqrt{3}$, $-8 \sqrt{3}$, $-8\sqrt{3}$, $8\sqrt{3}$,$8\sqrt{3}$, 
$8\sqrt{3}$, $-2\sqrt{7(45-8\sqrt{30})}$, 
$2\sqrt{7(45-8 \sqrt{30})}$, $-2\sqrt{3}$, $-2\sqrt{3}$, $-2\sqrt{3}$,
$-2\sqrt{3}$, $-2\sqrt{3}$, $2\sqrt{3}$, 
$2\sqrt{3}$, $2\sqrt{3}$, $2\sqrt{3}$, $2\sqrt{3}$, $0,0,0,0,0,0,0\}$. 

\section{Illustrations}

\begin{figure}[!htpb]
\scalebox{0.12}{\includegraphics{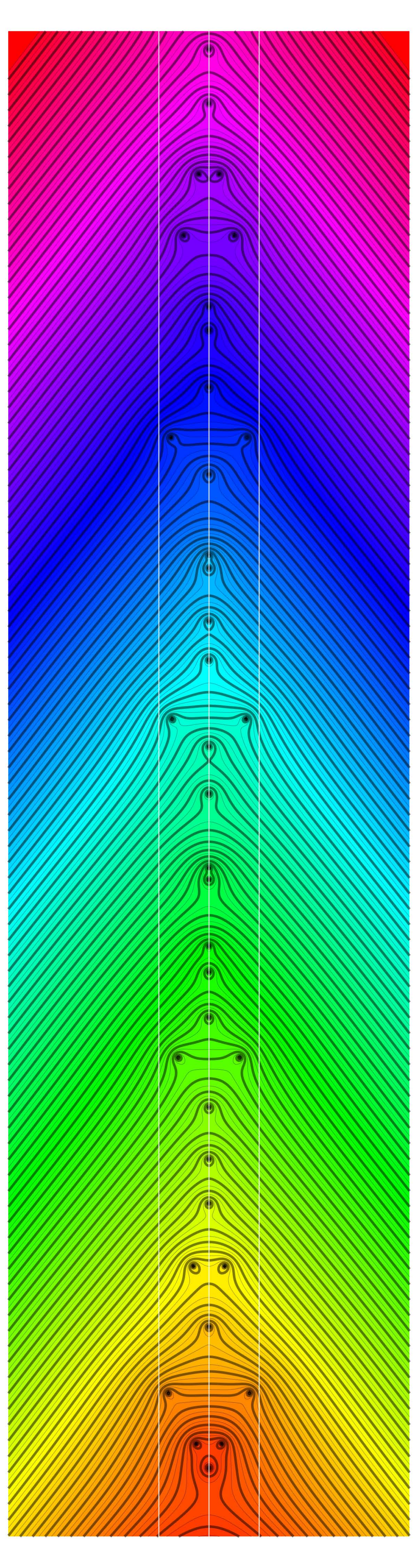}}
\label{zeta}
\caption{
Contour lines of the zeta function $\zeta(s)$ of a random complex.
In this case, the $f$-vector is $f= (10, 22, 13, 2)$. The level curves of 
$|\zeta(s)|$ are seen in the region
$\{ |{\rm Re}(s)| \leq 4$, $0\leq {\rm Im}(s) \leq 30 \}$.
The functional equation implies that the roots
are symmetric with respect to the imaginary axes. 
}
\end{figure}

\begin{figure}[!htpb]
\scalebox{0.12}{\includegraphics{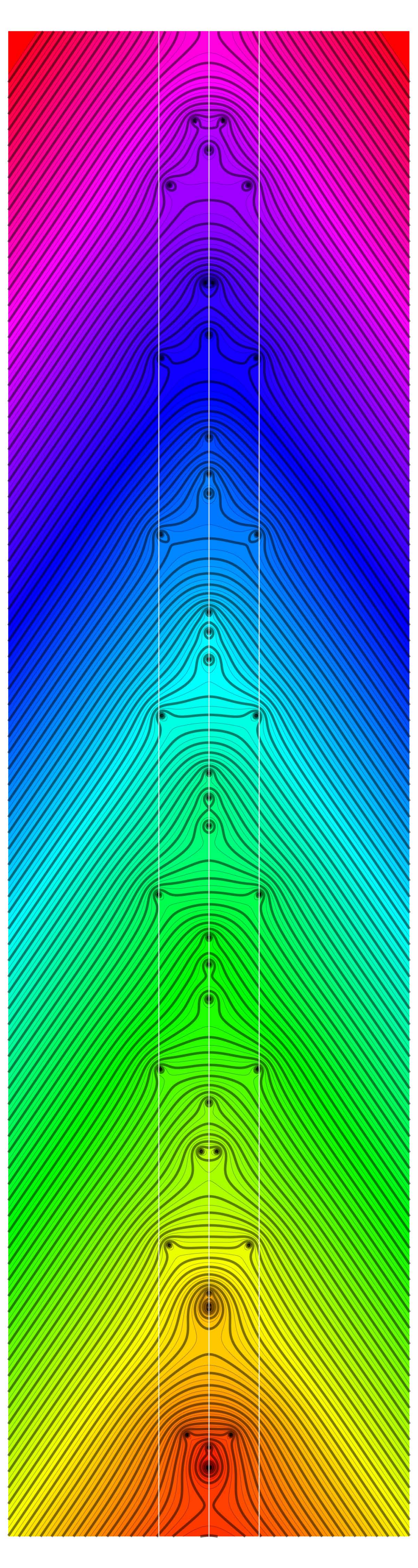}}
\label{zeta}
\caption{
Contour lines of the zeta function $\zeta(s)$ of the complete
complex $K_5$. 
Again, $|\zeta(s)|$ are seen in the region
$\{ |{\rm Re}(s)| \leq 4$, $0\leq {\rm Im}(s) \leq 30 \}$.
There are no results yet about the structure of the roots
even not in the special case of complete graphs. 
We know only the functional equation so far. 
}
\end{figure}

\begin{figure}[!htpb]
\scalebox{0.12}{\includegraphics{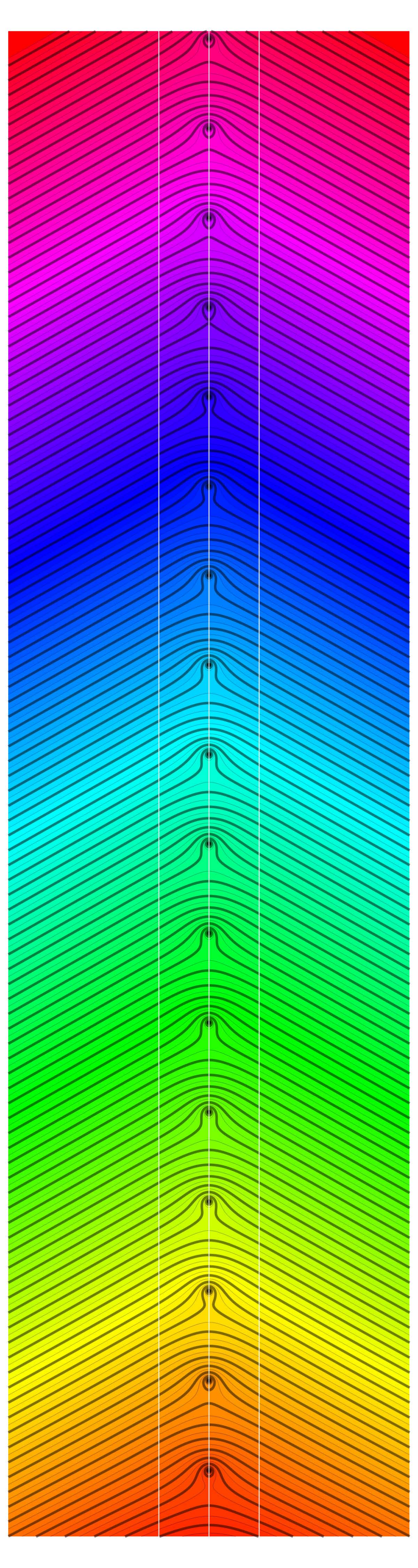}}
\label{zeta}
\caption{
Contour lines of the zeta function $\zeta(s)$ of the cyclic
complex $C_{40}$. The contours of $|\zeta(s)|$ are seen in the region
$\{ |{\rm Re}(s)| \leq 4$, $0\leq {\rm Im}(s) \leq 30 \}$.
In the pro-finite limit $n \to \infty$, the zeta function of a
one dimensional complex is explicit \cite{DyadicRiemann}. 
In two and higher dimensions we don't know the profinite limit. 
Also the universal density of state limit of $K$ is unexplored. 
}
\end{figure}

\begin{figure}[!htpb]
\scalebox{0.12}{\includegraphics{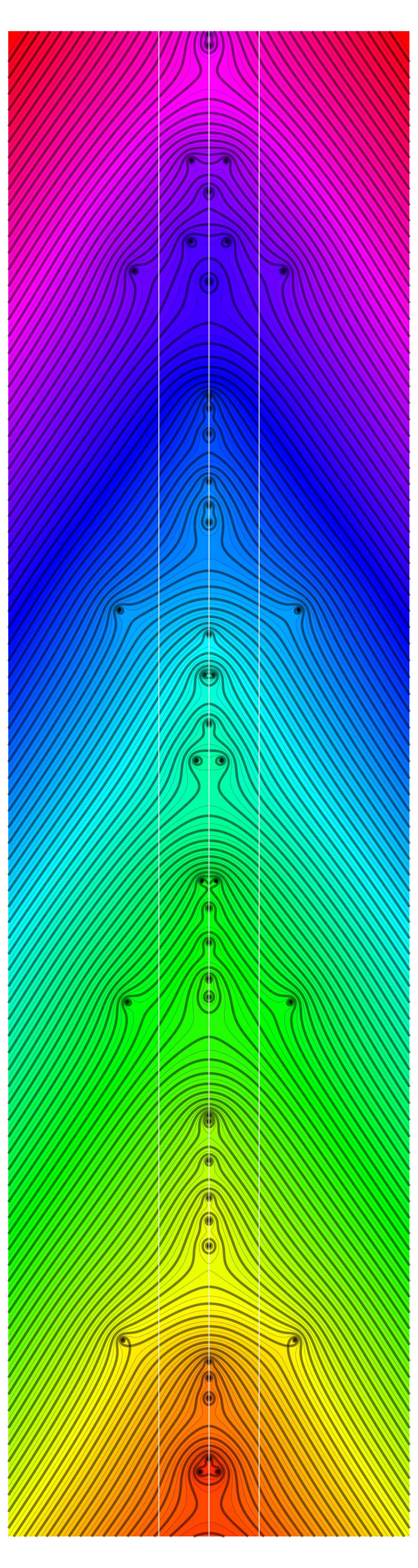}}
\label{zeta}
\caption{
Contour lines of the zeta function $\zeta(s)$ of the 3-sphere 
complex obtained by suspending the octahedron.
In this case, the $f$-vector is $f= (8, 24, 32, 16)$.
The function $|\zeta(s)|$ is again seen in the region
$\{ |{\rm Re}(s)| \leq 4$, $0\leq {\rm Im}(s) \leq 30 \}$.
}
\end{figure}

\begin{figure}[!htpb]
\scalebox{0.35}{\includegraphics{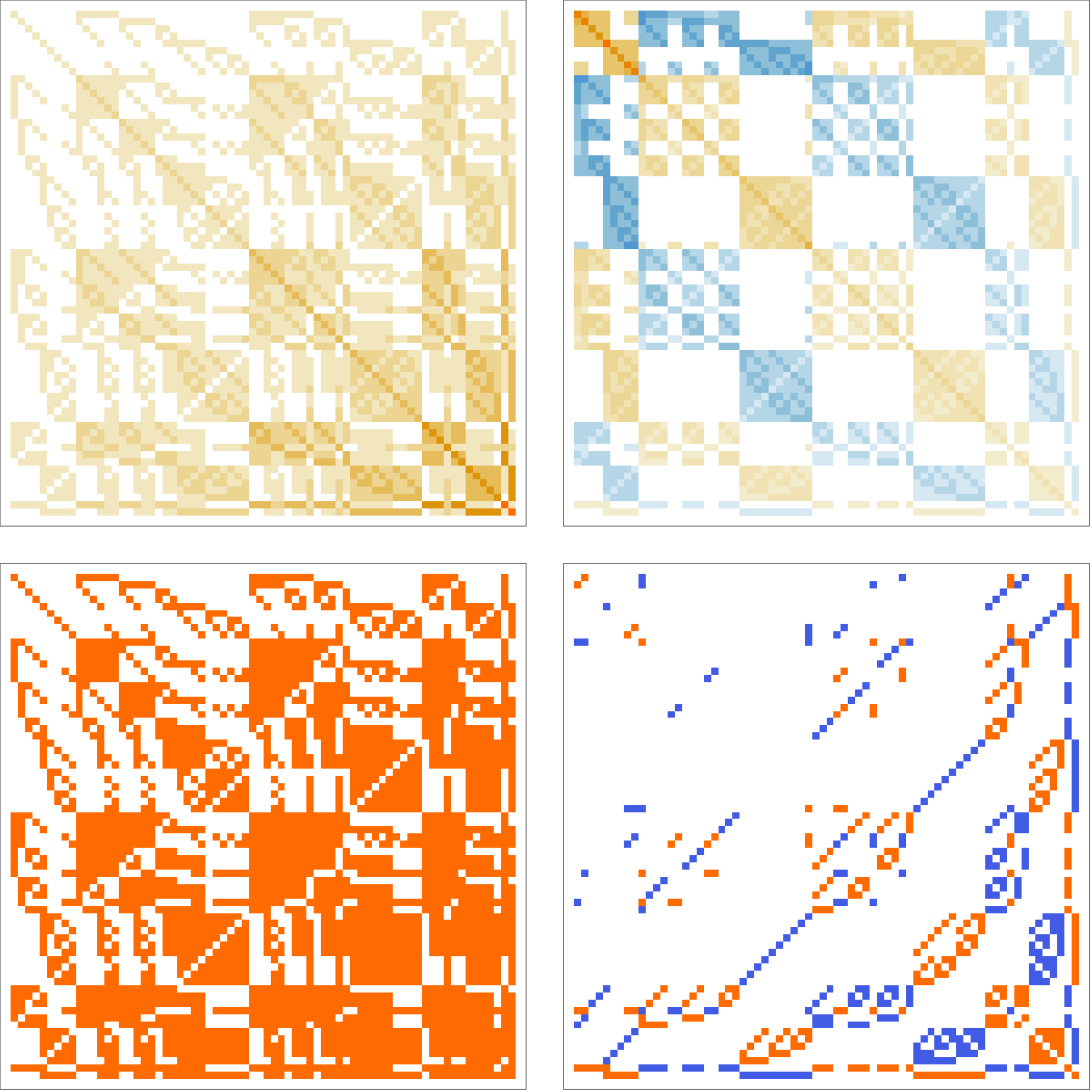}}
\label{matrixplot}
\caption{
The matrices $K,K^{-1}$ and as comparison, the matrices $L,L^{-1}$ 
are seen below. The complex $G$ is generated by 
$A=\{\{1,2,3,4,5\},\{5,6,7,8,9\},\{1,2,8,9\}\}$ and contains $70$ sets.
The $f$-vector is $(9, 24, 24, 11, 2)$. 
The $70 \times 70$ counting matrix $K$ has determinant $1$, the
connection matrix $L$ has determinant $-1$. There are 35 odd dimensional
sets. The matrix has $35$ positive eigenvalues, the matrix $K$ has $70$ 
positive eigenvalues in $[0.00868721,115.112]$. 
}
\end{figure}

\section{Remarks} 

\paragraph{}
In the one dimensional case, there appear relations between $L$ and $K$.
In the cyclic case for example $K-L$ is a direct sum of a $f_0 \times f_0$ zero matrix
$0$ with a $f_1 \times f_1$ constant diagonal matrix $2I$. In the one-dimensional
case, the kernel of $Q=K-K^{-1}$ and $H=L-L^{-1}$ are the same which makes the 
nullity of $Q$ topological in that case. We don't see yet whether it is possible 
to extract cohomology from the spectrum of $Q$ in general. 
See \cite{HearingEulerCharacteristic} for the case of the connection matrix $L$ in the
one dimensional case.

\paragraph{}
In the Barycentric limit, the density of states of $K$ converges to a measure
which only depends on the dimension \cite{KnillBarycentric2}. In the 
one-dimensional case, the limit is understood. 
While the limiting zeta function of the Hodge Laplacian is more
difficult to describe \cite{KnillZeta}, the Barycentric limiting case of $L$ or
$K$ is explicit \cite{DyadicRiemann}. As the matrices $K$ have more
spectral symmetry than $L$ it would be nice to understand the limiting root
structure of the zeta function in the Barycentric limit. 

\paragraph{}
The Riemann zeta function is the spectral function the circle $G=\mathbb{T}$.
The generator of translation $D=i \partial_x$ on $G$ has the eigenvalues $\mathcal{Z}$
with eigenfunctions $e^{i n x}$ to the eigenvalue $-n$. 
The Laplacian $H=D^2=-\Delta$ with eigenvalues $n^2$. One discards the singular harmonic 
case $\lambda=0$, and replaces $\sum_{n>0} (n^2)^{-s}$ with
$\sum_{n>0} n^{-s}$ so that rather than ${\rm Re}(s)=1$, the critical line is
${\rm Re}(s)=1/2$. Discarding $\lambda_k=0$ for Riemannian manifolds
allows to zeta regularize determinants. In the discrete case, when looking at the Dirac
operator, the zeta regularization is the pseudo-determinant of the matrix, the product of the
non-zero eigenvalues. The connection matrix and now the counting matrix case
are remarkable in that no regularization is needed. 

\paragraph{}
The counting Laplacian results for $K$ covered here combines with the connection Laplacian case
for $L$. Define the {\bf $f$-function} $f_G(t) = 1+\sum_{k=0}^d f_k t^{k+1}$, where $f_k$ is the 
number of $k$-dimensional simplices in $G$. Some results for Euler characteristic
generalize to the $f$-function. An example is {\bf parametrized Gauss-Bonnet} 
\cite{dehnsommervillegaussbonnet} 
telling that $f_G(t) = 1+ \sum_x F_{S(x)}(t)$, where $F_G(t)$ is the anti-derivative of $f_G(t)$. 
An other is the {\bf parametrized Poincar\'e-Hopf} \cite{parametrizedpoincarehopf} 
which assures that that for a locally 
injective function $g$ on the vertex set of a  finite simple graph, 
$f_G(t) = 1+t \sum_{x} f_{S_g(x)}(t)$, where $S_g(x) = \{ y \in S(x), g(y)<g(x) \}$ and
$S(x)$ is the set of vertices attached to $x$. There are exactly two unimodular cases among 
$L_t$: the case $t=-1$ leads to $L$ and the case $t=1$ which leads to $K$.

\paragraph{}
So, here is an announcement of the general case:
define $L_{t}(x,y) = (1-f_{W^-(x) \cap W^-(y)}(t))/t^{{\rm dim}(x \cap y)}$ which 
is rational in $t$. The inverse of $L_t$ is the Green function matrix 
$g_t(x,y) = \omega(x) \omega(y) (1-f_{W^+(x) \cap W^+(y)}(t))$.
The matrix $L_t$ has the determinant
$(-1)^{|G|}t^{f_G'(1)}$ and the total potential energy satisfies
$1-f_G(t)=\sum_{x,y \in G} g_t(x,y)$. The case $t=-1$ is the connection case, 
the case $t=1$ is the counting case $K=-L_{-1}$ discussed here as then 
${\rm det}(L_1)={\rm det}(-1)$ and ${\rm det}(K)=1$. 
We hope to be able to elaborate on this general case elsewhere. 

\vfill

\pagebreak

\section{Code}

\paragraph{}
The following Mathematica code generates the matrix $K$ and its inverse $K^{-1}$,
and the zeta function for a random complex according to the definitions and 
illustrates the results in examples. As usual, the code can be grabbed from 
the ArXiv. It should serve as pseudo code also: 

\begin{tiny}
\lstset{language=Mathematica} \lstset{frameround=fttt}
\begin{lstlisting}[frame=single]
Generate[A_]:=Delete[Union[Sort[Flatten[Map[Subsets,A],1]]],1];
R[n_,m_]:=Module[{A={},X=Range[n],k},Do[k:=1+Random[Integer,n-1];
  A=Append[A,Union[RandomChoice[X,k]]],{m}];Generate[A]];
G=R[6,9];n=Length[G]; G=Sort[G];                    w[x_]:=-(-1)^Length[x];
star[x_]:=Module[{u={}},Do[v=G[[k]];If[SubsetQ[v,x],u=Append[u,v]],{k,n}];u];
core[x_]:=Module[{u={}},Do[v=G[[k]];If[SubsetQ[x,v],u=Append[u,v]],{k,n}];u];
Wminus=Table[Intersection[core[G[[k]]],core[G[[l]]]],{k,n},{l,n}];
Wplus =Table[Intersection[star[G[[k]]],star[G[[l]]]],{k,n},{l,n}];
K = Table[                    Length[Wminus[[k,l]]],{k,n},{l,n}];
KI= Table[w[G[[k]]]*w[G[[l]]]*Length[Wplus[[k,l]]] ,{k,n},{l,n}];
EV = Sort[Eigenvalues[1.0*K]]; Clear[s]; ZetaFunction=Total[EV^(-s)];
CharPol=CoefficientList[CharacteristicPolynomial[K,s],s];
Print["Green Star formula: ",Simplify[K.KI==IdentityMatrix[n]]];
Print["Energy Theorem:     ",Total[Flatten[KI]]==Length[G]];
Print["Spectral Symmetry:  ",CharPol==(-1)^n*Reverse[CharPol]];    
\end{lstlisting}
\end{tiny}

\bibliographystyle{plain}

\end{document}